\documentclass[10pt]{amsart}
\usepackage[matrix,arrow]{xy}
\usepackage{amssymb}
\usepackage{amsmath}
\usepackage{amsfonts}
\usepackage{epsfig}
\usepackage{color}
\usepackage{epstopdf}
\usepackage{graphicx}
\usepackage{amsthm}
\usepackage{enumerate}
\usepackage[mathscr]{eucal}
\usepackage{verbatim}
\usepackage{bm}
\usepackage{mathtools}

\makeatletter
\g@addto@macro\bfseries{\boldmath} 
\makeatother

\usepackage[bookmarks=true,hyperindex,pdftex,colorlinks,citecolor=blue,
linkcolor=blue,urlcolor=blue]{hyperref}
\usepackage{tikz}
\usetikzlibrary{matrix,arrows.meta}

\usepackage{orcidlink}

\theoremstyle{plain}
\newtheorem{theorem}{Theorem}[section]
\newtheorem{cor}[theorem]{Corollary}
\newtheorem{prop}[theorem]{Proposition}
\newtheorem{lemma}[theorem]{Lemma}

\theoremstyle{definition}

\newtheorem{example}[theorem]{Example}

\newtheorem{rem}[theorem]{Remark}

\newtheorem{definition}[theorem]{Definition}

\newcommand{\N}{\mathbb{N}}

\DeclareMathOperator{\co}{co}

\DeclareMathOperator{\NA}{NA}

\renewcommand{\subset}{\subseteq}

\newcommand{\pten}{\ensuremath{\widehat{\otimes}_\pi}}

\renewcommand{\leq}{\leqslant}
\renewcommand{\geq}{\geqslant}

\title[Projective norm-attainments and their implications]
{Projective norm-attainments and their implications} 

\author[M.~Han]{Manwook Han\orcidlink{0009-0000-7250-1814}}
\address[M.~Han]{Department of Mathematics, Chungbuk National University, Cheongju, Chungbuk 28644, Republic of Korea\newline
	\href{https://orcid.org/0009-0000-7250-1814}{ORCID: \texttt{0009-0000-7250-1814}  }}
\email{\texttt{mwhan0828@gmail.com}}

\author[S.~K.~Kim]{Sun Kwang Kim\orcidlink{0000-0002-9402-2002}}
\address[S.~K.~Kim]{Department of Mathematics, Chungbuk National University, Cheongju, Chungbuk 28644, Republic of Korea\newline
	\href{http://orcid.org/0000-0002-9402-2002}{ORCID: \texttt{0000-0002-9402-2002}  }}
\email{\texttt{skk@chungbuk.ac.kr}}

\author[M.~Mart\'in]{Miguel Mart\'in\orcidlink{0000-0003-4502-798X}}
\address[M.~Mart\'in]{Department of Mathematical Analysis and Institute of Mathematics (IMAG), University of Granada, 18071 Granada, Spain \newline \href{http://orcid.org/0000-0003-4502-798X}{ORCID: \texttt{0000-0003-4502-798X} } }
\email{\texttt{mmartins@ugr.es}}
\urladdr{\url{http://www.ugr.es/~mmartins/}}

\author[A.~Rueda Zoca]{Abraham Rueda Zoca\orcidlink{0000-0003-0718-1353}}
\address[A.~Rueda Zoca]{Department of Mathematical Analysis and Institute of Mathematics (IMAG), University of Granada, 18071 Granada, Spain \newline
	\href{https://orcid.org/0000-0003-0718-1353}{ORCID: \texttt{0000-0003-0718-1353}  }}
\email{\texttt{abrahamrueda@ugr.es}}

\urladdr{\url{https://arzenglish.wordpress.com}}


\date{\today}
\begin{document}

\begin{abstract}
We show that nuclear norm-attaining operators (resp.\ polynomials) are always $w^*$-dense in the space of integral operators (resp.\ polynomials). Besides, the denseness is in norm if the predual space does not contain any isomorphic copy of $\ell_1$. We also show that there are reflexive spaces for which the set of projective norm-attaining elements does not coincide with the whole projective tensor product (which is indeed also reflexive here). 
Next, we show that if $Y$ is a II-polyhedral space, then every nuclear operator from an arbitrary space $X$ to $Y^*$ attains its nuclear norm. As a consequence, if $X^*$ or $Y^*$ has the approximation property, then the set of norm-attaining operators from $X^*$ to $Y^{**}$ is dense. Finally, we study proximinality results of a natural subspace of the projective tensor product and obtain an application to integral projective norm-attaining tensors which solves a proposed open question.
\end{abstract}

\maketitle

\section{Introduction}\label{Section 1}

Let $X$ and $Y$ be Banach spaces over the scalar field $\mathbb{K}$, where $\mathbb{K}$ is either $\mathbb{R}$ or $\mathbb{C}$. All topological notions in the paper refer to the norm topology unless explicitly indicated. The space of all bounded linear operators from $X$ to $Y$ is denoted by $\mathcal{L}(X,Y)$, and we write $X^*=\mathcal{L}(X,\mathbb{K})$ for the topological dual of $X$. The closed unit ball and the unit sphere of the space $X$ are denoted, respectively, by $B_X$ and $S_X$. 
An operator $T\in\mathcal{L}(X,Y)$ is said to \emph{attain its norm} if there exists an element $x\in S_X$ such that $\|Tx\|=\|T\|$.
Such an operator $T$ is referred to as a \emph{norm-attaining operator}, and the set of all norm-attaining operators from $X$ to $Y$ is denoted by $\textrm{NA}(X,Y)$.
More generally, the question of whether the supremum or infimum defining a given norm is actually achieved, arises naturally in many contexts throughout Banach space theory, and the study of norm-attaining (linear) operators is one of the earliest and most fruitful instances of this philosophy.

In this context, one of the definitive results is James's theorem \cite{Jam}, which establishes that a Banach space $X$ is reflexive if and only if every element in its dual space $X^*$ attains its norm.
Furthermore, the Bishop--Phelps theorem \cite{BP} guarantees that for every Banach space $X$, the set of norm-attaining functionals is always norm-dense in $X^*$.

These results naturally lead to the investigation of the abundance of norm-attaining operators.
In his seminal work \cite{Lind}, Lindenstrauss demonstrated that the Bishop--Phelps theorem cannot be extended to vector-valued operators, that is, $\mathrm{NA}(X,Y)$ is not always dense in $\mathcal{L}(X,Y)$.
However, he proved a major positive counterpart: if the domain $X$ is reflexive, then it has property $A$ that means $\mathrm{NA}(X,Y)$ remains dense in $\mathcal{L}(X,Y)$ for any $Y$, and the same happens, for instance, if $X=\ell_1$. On the range space side, it is shown in the same paper that $\mathrm{NA}(X,Y)$ is dense if $Y$ is a finite-dimensional polyhedral space or a subspace of $\ell_\infty$ containing the canonical copy of $c_0$, among many other examples. A striking result was achieved by Bourgain \cite{Bour}, who proved that every Banach space with the Radon--Nikod\'ym property (RNP) has property $A$, and conversely, if every equivalent renorming of $X$ has property $A$, then $X$ has the RNP (this last result needs a refinement due to Huff \cite{Huff}). Since then, the interplay between geometric properties of unit balls (such as dentability, extreme point structure, and polyhedrality) and the abundance of norm-attaining operators has remained a central theme in Banach space theory. We refer the reader to the paper \cite{acostaRACSAM} for an account of all classical results and to \cite{residual} and references therein for more recent ones.

An effective framework for studying the spaces of bounded linear operators between Banach spaces is provided by the theory of tensor products.
Given Banach spaces $X$ and $Y$, the projective tensor product $X\widehat\otimes_\pi Y$ and the injective tensor product $X\widehat\otimes_\varepsilon Y$ are the completions of the algebraic tensor product $X\otimes Y$ under two natural norms: 
$$\|z\|_\pi\coloneqq\inf\left\{\sum_{i=1}^n\|x_i\|\|y_i\|\colon z=\sum_{i=1}^nx_i\otimes y_i \right\}$$
and
$$\|z\|_\varepsilon\coloneqq\sup\left\{\left|\sum_{i=1}^nx^*(x_i)y^*(y_i)\right|\colon  x^*\in B_{X^*},~y^*\in B_{Y^*},~\text{and}~z=\sum_{i=1}^nx_i\otimes y_i \right\},$$
respectively. Here, the infimum and supremum are taken over all possible representations of $z\in X\otimes Y$.
These two constructions encode operators in dual ways.
On the one hand, the dual space of the projective tensor product $\left(X\widehat\otimes_\pi Y\right)^*$ is isometrically isomorphic to the space of bounded linear operators $\mathcal{L}(X,Y^*)$ \cite[Chapter 2.2]{Rya}.
On the other hand, every elementary tensor $x\otimes y$ induces a rank-one operator $x^*\mapsto x^*(x)y$ from $X^*$ to $Y$, and this identification extends to an isometric embedding of $X\widehat\otimes_\varepsilon Y$ into $\mathcal{L}(X^*,Y)$ whose image coincides with the closure of the set of all finite-rank operators \cite[Chapter 3.1]{Rya}. Consequently, under the above point of view, $\Vert\cdot\Vert_\varepsilon$ is nothing but the operator norm.

One may ask the same question for a tensor as classical theory asks for an operator. Following \cite{DJRR}, we say that an element $z\in X\widehat\otimes_\pi Y$ \emph{attains its projective norm} if the infimum defining $\|z\|_\pi$ is realized by some representation; that is, if there exist sequences $(x_i)_i$ in $X$ and $(y_i)_i$ in $Y$ such that
$$\|z\|_\pi=\sum_{i=1}^\infty \|x_i\|\|y_i\|\quad\text{and}\quad z=\sum_{i=1}^\infty x_i\otimes y_i.$$
We write $\mathrm{NA}\left(X\widehat\otimes_\pi Y\right)$ for the set of all such elements.
This is the exact tensor product counterpart of classical norm-attainment.

Far from being a merely formal analogue, this notion has direct consequences in the classical theory.
Indeed, it was shown in \cite[Corollary 3.11]{DJRR} that if every element of $X\widehat\otimes_\pi Y$ attains its projective norm, then $\mathrm{NA}(X,Y^*)$ is dense in $\mathcal{L}(X,Y^*)$.
The same circle of ideas admits a counterpart for symmetric tensor products and homogeneous polynomials; we defer the precise formulations to Subsection \ref{Subsection 2}.

The question of how abundant norm-attaining tensors are has been studied from two complementary angles: density and full-attainment.
On the density side, Dantas, Jung, Rold\'an, and the last author \cite[Theorem 4.8]{DJRR} proved that $\mathrm{NA}\left(X\widehat\otimes_\pi Y\right)$ is dense in $X\widehat\otimes_\pi Y$ whenever both spaces have the metric $\pi$-property or one of them has the metric $\pi$-property and the other is uniformly convex.
Since $L_p(\mu)$-spaces, $L_1$-preduals, and spaces with a monotone finite-dimensional decomposition all satisfy the metric $\pi$-property, this covers most classical spaces.
They also exhibited Banach spaces $X$ and $Y$, both failing the {approximation property (AP)}, for which $\mathrm{NA}\left(X\widehat\otimes_\pi Y^*\right)$ is \emph{not} dense in $X\widehat\otimes_\pi Y^*$, demonstrating that the AP plays an essential role in these density results.
In a related direction, Dantas, Garc\'ia-Lirola, Jung, and the last author \cite[Theorem 4.7]{DGJR} proved that $\mathrm{NA}\left(X^*\widehat\otimes_\pi Y^*\right)$ is $w^*$-dense in $X^*\widehat\otimes_\pi Y^*$ whenever $X^*$ or $Y^*$ has the RNP and $X^*$ or $Y^*$ has the AP.

On the full-attainment side (that is, on the question of when every element in a projective tensor product attains its projective norm), the major contribution was done by Garc\'ia-Lirola, Guerrero-Viu, and the last author in \cite{GGR} (even though there are previous examples in the literature as \cite[Propositions 3.5, 3.6, and 3.8]{DJRR} and \cite[Theorem 4.1]{DGJR}). They proved that if $X$ is a subspace of an $\ell_1$-predual space, and $Y$ is $1$-complemented in $Y^{**}$ (in particular, if $Y$ is a dual space), then $\mathrm{NA}\left(X^*\widehat\otimes_\pi Y\right)=X^*\widehat\otimes_\pi Y$, provided that $X^*$ or $Y^{**}$ has the AP.

The projective and injective tensor products provide the appropriate setting to introduce nuclear and integral operators on Banach spaces, to which we now turn. An operator $T\in\mathcal{L}(X,Y)$ is called \emph{nuclear} if it admits a representation $T=\sum_{i=1}^\infty x^*_i\otimes y_i$ with $x^*_i\in X^*$ and $y_i\in Y$ for every $i\in \N$ satisfying that $\sum_{i=1}^\infty \|x^*_i\|\|y_i\|<\infty$. The \emph{nuclear norm} of $T$ is defined by
$$\|T\|_\text{nu}\coloneqq\inf\left\{\sum_{i=1}^\infty\|x_i^*\|\|y_i\|\colon T=\sum_{i=1}^\infty x_i^*\otimes y_i\right\},$$
where the infimum is taken over all such representations.
We write $\mathcal{N}(X,Y)$ for the space of nuclear operators equipped with $\|\cdot\|_\mathrm{nu}$ and $\mathrm{NA}_\mathrm{nu}(X,Y)$ for those nuclear operators for which the infimum above is actually a minimum.  
The natural map $J\colon X^*\widehat\otimes_\pi Y\to \mathcal{N}(X,Y)$ defined by 
$$J\left(\sum_{i=1}^\infty x^*_i\otimes y_i\right)\colon x\mapsto \sum_{i=1}^\infty x_i^*(x)y_i\quad \forall x\in X,$$
is a norm-one quotient map, and induces an isometric isomorphism $\mathcal{N}(X,Y)\equiv \left(X^*\widehat\otimes_\pi Y\right)/\mathrm{ker}(J)$ \cite[Chapter 2.6]{Rya}. Grothendieck showed that $J$ is an isometric isomorphism (equivalently, $\mathrm{ker}(J)=\{0\}$) if and only if one of the spaces $X^*$ or $Y$ has the AP (see \cite[Proposition 4.6]{Rya}). In such a case, statements phrased in terms of the projective tensor product and in terms of nuclear operators coincide. Let us comment that, being $\ker(J)=0$ or not, all the positive results about density of projective norm-attaining tensors from \cite{DGJR,DJRR} have a corresponding counterpart for the density of nuclear operators which attain their nuclear norm which are actually stated in the same papers. 

In a dual way, the injective tensor product embeds isometrically into a space of continuous functions. To be more precise, the map $\Phi\colon X\widehat\otimes_\varepsilon Y\hookrightarrow C\left(B_{X^*}\times B_{Y^*},w^*\times w^*\right)$ defined by $$\Phi\left(\sum_{i=1}^\infty x_i\otimes y_i\right)(x^*,y^*)=\sum_{i=1}^\infty x^*(x_i)\, y^*(y_i)\quad \forall (x^*,y^*)\in B_{X^*}\times B_{Y^*}$$
is an isometry.
Since the dual of the space $C(K)$ of continuous functions on a compact Hausdorff space $K$ is the space $\mathcal{M}(K)$ of Radon measures on $K$, every bounded functional on $X\widehat\otimes_\varepsilon Y$ extends to integration against one such measure. (By a \emph{Radon measure} on $K$, we mean a regular Borel measure, with total variation norm $|\mu|_\mathrm{TV}$, which we identify with an element of $C(K)^*$ via the Riesz representation theorem.) Accordingly, an operator $T\in\mathcal{L}(X,Y^*)$ is called \emph{integral} if there exists a Radon measure $\mu$ on $(B_{X^*}\times B_{Y^*},w^*\times w^*)$ such that
$$\langle Tx,y\rangle=\int_{B_{X^*}\times B_{Y^*}}x^*(x)y^*(y)\,d\mu(x^*,y^*)\quad \forall (x,y)\in X\times Y.$$
In this case, we say that the measure $\mu$ \emph{represents} $T$. The integral norm of $T$ is defined by $\|T\|_\mathcal{I}\coloneqq\inf|\mu|_\mathrm{TV}$, where the infimum is taken over all measures $\mu$ representing $T$.
We write $\mathcal{I}(X,Y^*)$ for the space of integral operators, equipped with $\|\cdot\|_\mathcal{I}$.
The integral operators are precisely the bounded functionals on the injective tensor product: one has the canonical isometric identification $\mathcal{I}(X,Y^*)\equiv \left(X\widehat\otimes_\varepsilon Y\right)^*$ \cite[Chapter 3.5]{Rya}.
Every nuclear operator is integral with $\|T\|_\mathcal{I}\leq\|T\|_\mathrm{nu}$, in other words, we have that $\mathcal{N}(X,Y^*)\subset\mathcal{I}(X,Y^*)$ contractively. At the end of this introduction, Subsection~\ref{Subsection 2} contains analogous definitions to the above ones for the multilinear and polynomial case.

Let us observe that in each of the results exposed above, the conclusion is obtained under a structural hypothesis as the metric $\pi$-property, the RNP or the AP. It is natural to ask how far such hypotheses can be weakened.
Our aim in this paper is to establish more general density and full-attainment results with minimum requirements. 

Let us outline the content of the paper. The results on denseness are contained in Section~\ref{Section 3}. The main tool to get the results here is an abstract density principle valid for every Banach space and dealing with extreme points (Lemma~\ref{gen}). Applying this result to tensor products via the description of the extreme point structure of the dual unit ball, we get that the set $\mathrm{NA}_\mathrm{nu}(X,Y)$ of nuclear norm-attaining elements is $w^*$-dense in the dual $\mathcal{I}(X,Y^*)$ of the injective tensor product $X\widehat\otimes_\varepsilon Y$ for arbitrary Banach spaces $X$ and $Y$ (Theorem~\ref{cor:wstar-density}). Besides, they are norm-dense as soon as $X\widehat\otimes_\varepsilon Y$ does not contain any isomorphic copy of $\ell_1$ (Theorem \ref{cor:wstar-density} again). If, in addition, either $X^*$ or $Y^*$ has the AP, it follows that the set of projective norm-attaining tensors is dense in $X^*\widehat\otimes_\pi Y^*$ (Corollary~\ref{Corollary:onedualAP-NAtensors-wstar-dense}). The results above improve previous ones from \cite{DGJR} and \cite{Ros2}. Analogous statements hold for symmetric tensor products (Corollary \ref{cor:polywstar-density}) which also improve results from \cite[Section 3]{DGJR}. Finally, we also show that $\ell_p\widehat\otimes_\pi\ell_q$ contains a non projective norm-attaining element for $1<p,q<\infty$ with $\frac{1}{p}+\frac{1}{q}<1$ (Example~\ref{Example}), being the first example of this kind involving reflexive spaces.

The second line of results of this paper, contained in Section~\ref{Section 4}, is related to the study of a geometric condition under which every element of the projective tensor product attains its projective norm.
Recall that a Banach space $X$ is \emph{II-polyhedral} if there is $0\leq r<1$ such that the $w^*$-cluster points of the set $\mathrm{Ext}(B_{X^*})$ lie in $rB_{X^*}$  (see \cite{FV} and references therein for background). The main result of the section (Theorem~\ref{theoremB}) shows that given a II-polyhedral Banach space $X$ and an arbitrary Banach space $Y$ such that one of the spaces $X^*$ or $Y^*$ has the AP, then 
$\mathrm{NA}\left(X^*\widehat\otimes_\pi Y^*\right)=X^*\widehat\otimes_\pi Y^*$. As a consequence, using \cite[Corollary 3.11]{DJRR}, we get that $\overline{\mathrm{NA}(X^*,Y^{**})}=\mathcal{L}(X^*,Y^{**})$ and $\overline{\mathrm{NA}(Y^*,X^{**})}=\mathcal{L}(Y^*,X^{**})$ under the same hypotheses (Corollary~\ref{cor:lineardense}), providing new examples on which norm-attaining operators are dense. The proofs in this section are based on a measure-theoretic decomposition that reduces an arbitrary representing measure to an atomic one supported on the extreme points, and the argument extends to the multilinear setting (Theorem \ref{main}).

Finally, in Section~\ref{Section 5} we examine structural consequences of nuclear norm-attainment for $\ker(J)$ when viewed as a subspace of the projective tensor product. On the one hand, for general Banach spaces $X$ and $Y$, the subspace $\ker(J)$ is proximinal in $X\widehat\otimes_\pi Y$ whenever $\mathrm{NA}_\mathrm{nu}(X,Y)=\mathcal{N}(X,Y)$ (Proposition~\ref{proximinal}). On the other hand, if $X^*$ and $Y^*$ are separable, then $\ker(J)$ is automatically proximinal in $X^*\widehat\otimes_\pi Y^*$ (Theorem~\ref{theo:mainproxi}). As a by-product, we obtain an improvement of \cite[Proposition~3.2]{ADGJR} on integral projective norm-attaining tensors (Corollary~\ref{cor:INA}). This last result, taken together with Example~\ref{Example}, gives a negative answer to \cite[Question~2.2]{ADGJR}.

\subsection{Notation for multilinear forms and polynomials}\label{Subsection 2}
All the framework we have presented in the linear case extends to the multilinear setting as follows.
For $N\in\mathbb{N}$ and Banach spaces $X_1,\ldots,X_N$, let $\otimes_{i=1}^NX_i$ denote their $N$-fold algebraic tensor product, spanned by the elementary tensors $\otimes_{i=1}^N x_i$, $x_i\in X_i$.
Its completion under the projective norm 
$$\|z\|_\pi\coloneqq\inf\left\{\sum_{j=1}^n\prod_{i=1}^N\|x_{i,j}\|\colon z=\sum_{j=1}^n\otimes_{i=1}^Nx_{i,j}\right\}$$
is the \emph{$N$-fold projective tensor product} $\widehat\otimes_{\pi,i=1}^N X_i$, and its completion under the injective norm 
$$\|z\|_\varepsilon\coloneqq\sup\left\{\left|\sum_{j=1}^n\prod_{i=1}^Nx_i^*(x_{i,j})\right|\colon x_i^*\in B_{X_i^*}~\text{and}~z=\sum_{j=1}^n \otimes_{i=1}^Nx_{i,j} \right\}$$
is the \emph{$N$-fold injective tensor product} $\widehat\otimes_{\varepsilon,i=1}^N X_i$.
Just as $\widehat\otimes_\pi$ and $\widehat\otimes_\varepsilon$ encode operators, these encode $N$-linear forms: writing $\mathcal{B}\left(\prod_{i=1}^N X_i\right)$ for the space of bounded $N$-linear forms on $X_1\times\cdots\times X_N$, one has the canonical isometric duality $\left(\widehat\otimes_{\pi,i=1}^N X_i\right)^*\equiv\mathcal{B}\left(\prod_{i=1}^N X_i\right)$, while $\widehat\otimes_{\varepsilon,i=1}^NX_i$ embeds isometrically into $C\left(\prod_{i=1}^NB_{X^*_i},\prod_{i=1}^N w^*\right)$ via 
$$\Phi\left(\sum_{j=1}^n\otimes_{i=1}^N x_{i,j}\right)\left((x_i^*)_i\right)=\sum_{j=1}^n\prod_{i=1}^Nx^*_i(x_{i,j})\quad \forall (x_i^*)_{i=1}^N\in \prod_{i=1}^NB_{X^*_i}.$$
Similarly to integral operators, we consider the space $\mathcal{I}\left(\prod_{i=1}^N{X_i}\right)$ of integral multilinear forms  on $\prod_{i=1}^N X_i$ with the canonical isometric equivalence $\left(\widehat\otimes_{\varepsilon,i=1}^NX_i\right)^*\equiv\mathcal{I}\left(\prod_{i=1}^N{X_i}\right)$. This space is endowed with the integral norm $\Vert{}\cdot \Vert{}_\mathcal{I}$, which is defined as the infimum of the total variations over all measures representing the multilinear form.

Restricting to $X_1=\cdots=X_N=X$ we can consider the symmetric counterpart, in which $N$-linear forms become homogeneous polynomials.
Let $\otimes_{s,N}X\subset\otimes_{i=1}^N X$ be the symmetric $N$-fold algebraic tensor product, spanned by $x^N\coloneqq x\otimes\overset{N}\cdots\otimes x$ for every $x\in X$.
Here, the roles above are played by polynomials: writing $\mathcal{P}\left(^NX\right)$ for the space of $N$-homogeneous polynomials on $X$, the projective side gives the isometric isomorphism $\left(\widehat\otimes_{\pi,s,N}X\right)^*\equiv\mathcal{P}\left(^NX\right)$, while the injective side embeds isometrically into $C\left(B_{X^*},w^*\right)$, so that $\left(\widehat\otimes_{\varepsilon,s,N}X\right)^*\equiv\mathcal{P}_\mathcal{I}\left(^NX\right)$, the space of $N$-homogeneous integral polynomials on $X$ with the integral norm $\|\cdot\|_\mathcal{I}$ \cite{BR}.

An $N$-linear form $M\in\mathcal{B}\left(\prod_{i=1}^N X_i\right)$ is said to be \emph{nuclear} if there exists a functional $x_{i,j}^*\in X_i^*$ for each $(i,j)\in\{1,...,N\}\times \mathbb{N}$ such that
$$\sum_{j=1}^\infty\prod_{i=1}^N\|x_{i,j}^*\|<\infty \quad \text{and}\quad M\left((x_i)_{i=1}^N\right)=\sum_{j=1}^\infty\prod_{i=1}^N x^*_{i,j}(x_i)\quad \forall (x_i)_{i=1}^N\in \prod_{i=1}^N X_i,$$
and its nuclear norm $\|M\|_\mathrm{nu}$ is the infimum of $\sum_{j=1}^\infty\prod_{i=1}^N\|x_{i,j}^*\|$ over all such representations. In the symmetric case, a polynomial $P\in\mathcal{P}\left(^NX\right)$ is \emph{nuclear} if there exist sequences $(\varphi_i)_i$ in $S_{X^*}$ and $(\lambda_i)_i$ in $\mathbb{K}$, such that 
$$\sum_{i=1}^\infty|\lambda_i|<\infty \quad\text{and}\quad P(x)=\sum_{i=1}^\infty \lambda_i\varphi_i(x)^N\quad \forall x\in X,$$
and the nuclear norm $\|P\|_\mathrm{nu}$ is the infimum of $\sum_{i=1}^\infty|\lambda_i|$ over all such representations. We write $\mathcal{N}\left(\prod_{i=1}^N X_i\right)$ and $\mathcal{P}_\mathrm{nu}\left(^NX\right)$ for the corresponding spaces.

In both settings, every nuclear element is integral, with $\|\cdot\|_\mathcal{I}\leq\|\cdot\|_\mathrm{nu}$. We say that an element of $\mathcal{N}\left(\prod_{i=1}^N X_i\right)$ or $\mathcal{P}_\mathrm{nu}\left(^NX\right)$ \emph{attains its nuclear norm} if the infimum defining $\|\cdot\|_\mathrm{nu}$ is realized by some representation. We denote by $\mathrm{NA}_\mathrm{nu}\left(\prod_{i=1}^N X_i\right)$ and $\mathrm{NA}_\mathrm{nu}\left(^NX\right)$ the corresponding sets of nuclear norm-attaining elements.

Grothendieck's result on when $J$ is injective has an analogous form in the multilinear setting: if all but one of $X_1^*,\ldots,X_N^*$ have the AP, then $\widehat\otimes_{\pi,i=1}^NX_i^*\equiv\mathcal{N}\left(\prod_{i=1}^NX_i\right)$ isometrically and, in the symmetric case, if $X^*$ has the AP, then $\mathcal{P}_\mathrm{nu}\left(^NX\right)\equiv\widehat\otimes_{\pi,s,N} X^*$ isometrically \cite[p.~20]{BR}. Then, as happens in the linear case, all the statements on nuclear elements and on symmetric tensors coincide under this hypothesis.

\section{Density of nuclear norm-attaining operators}\label{Section 3}

Let $X$ and $Y$ be two Banach spaces.
In this section we will examine the following two problems:
\begin{itemize}
    \item When is $\mathrm{NA}_\mathrm{nu}(X,Y)$ $w^*$-dense in $\mathcal{I}(X,Y)$?
    \item When is $\mathrm{NA}_\mathrm{nu}\left(^NX\right)$ $w^*$-dense in $\mathcal{P}_{\mathcal{I}}\left(^NX\right)$?
\end{itemize}

Our motivation goes back to \cite[Theorems 3.11 and 4.7]{DGJR}.
On the one hand, it is proved in \cite[Theorem 4.7]{DGJR} that $\mathrm{NA}_\mathrm{nu}(X,Y^*)$ is $w^*$-dense in $\mathcal{I}(X,Y^*)$ if $X^*$ has the RNP. 
On the other hand, it is proved in \cite[Theorem 3.11]{DGJR} that $\mathrm{NA}_\mathrm{nu}\left(^NX\right)$ is $w^*$-dense in $\mathcal{P}_{\mathcal{I}}\left(^NX\right)$ if $\widehat\otimes_{\varepsilon,s,N}X$ does not contain any isomorphic copy of $\ell_1$.

To begin with, we will improve the above results by showing that the $w^*$-density above always holds.  Moreover, this density can be replaced by the norm-density whenever $X\widehat\otimes_\varepsilon Y$ and $\widehat\otimes_{\varepsilon,s,N}X$ does not contain any isomorphic copy of $\ell_1$, respectively. The result will follow from the following principle which is valid in a more general context. This also slightly improves the fact that the closed unit ball of the dual of a Banach space coincides with the $w^*$-closed convex hull of its extreme points which is an application of the Krein-Milman theorem.

\begin{lemma}[Main technical lemma]\label{gen}
    Let $X$ be a Banach space, and let
    $$A\coloneqq\left\{\sum_{i=1}^n \lambda_i e_i^* \colon e_i^*\in \mathrm{Ext}(B_{X^*}),~\left\|\sum_{i=1}^n\lambda_i e_i^*\right\|=\sum_{i=1}^n\lambda_i=1,~\lambda_i\geq 0\right\}.$$
    Then $A$ is $w^*$-dense in $S_{X^*}$.
    Moreover, if $X$ does not contain any isomorphic copy of $\ell_1$, then $A$ is norm-dense in $S_{X^*}$.
\end{lemma}

To prove the second part of the above lemma, we invoke \cite[Theorem 3.3]{Haydon} of Haydon (Lemma \ref{Hay} in the following). Although this result is originally stated for real Banach spaces, it holds for complex ones as well, since if a complex Banach space contains no isomorphic copy of $\ell_1$, then its underlying real Banach space does not contain any isomorphic copy of $\ell_1$ as well. This fact follows directly from Rosenthal’s $\ell_1$ theorem and its complex counterpart established by Dor, see \cite[Chapter~XI]{Diestel} for instance.

\begin{lemma}\label{Hay}
    If $X$ is a Banach space that contains no isomorphic copy of $\ell_1$, then every $w^*$-compact convex subset of $X^*$ is the norm closed convex hull of its extreme points.
\end{lemma}

\begin{proof}[Proof of Lemma~\ref{gen}]
Let $U$ be a $w^*$-open subset of $B_{X^*}$, which intersects 
$S_{X^*}$. By the Bishop-Phelps theorem, there exists $y^*\in S_{X^*}\cap U$ which attains its norm at some $y_0\in B_X$. Define a $w^*$-closed face of $B_{X^*}$ by
\begin{equation}\label{eq:definition-C}
 C\coloneqq\{x^*\in B_{X^*}\colon x^*(y_0)=1\}.   
\end{equation}
By the Krein-Milman theorem, $C=\overline{\text{co}}^{w^*}(\mathrm{Ext}(C))$. Since, by construction, $y^* \in C \cap U$, there exists an element in $\text{co}(\mathrm{Ext}(C)) \cap U$. We represent this element as a finite convex combination$$\sum_{i=1}^n \lambda_i e^*_i,$$where $e^*_i \in \mathrm{Ext}(C)$ and the coefficients $\lambda_i$ satisfy $\sum_{i=1}^n \lambda_i = 1$ with $\lambda_i \geq 0$ for all $i=1, \dots, n$.

Since $C$ is a face of $B_{X^*}$, we have $e_i^*\in \mathrm{Ext}(B_{X^*})$. Furthermore, we have$$\left\| \sum_{i=1}^n \lambda_i e_i^* \right\| \leq \sum_{i=1}^n \lambda_i = \sum_{i=1}^n \lambda_i e_i^*(y_0) \leq \left\| \sum_{i=1}^n \lambda_i e_i^* \right\|.$$As these inequalities must hold as equalities, it follows that $\sum_{i=1}^n \lambda_i e_i^* \in A$.

We now suppose that $X$ does not contain an isomorphic copy of $\ell_1$. For arbitrary $x^* \in S_{X^*}$ and $\varepsilon > 0$, the Bishop-Phelps theorem guarantees the existence of a functional $y^* \in S_{X^*}$ with $\|x^* - y^*\| < \frac{\varepsilon}{2}$ which attains its norm at some $y_0 \in S_X$.

Since the set $C$ defined in \eqref{eq:definition-C} is $w^*$-compact and convex, Lemma \ref{Hay} yields $C=\overline{\co}(\mathrm{Ext}(C))$.
Therefore, there exists a finite convex combination $$\sum_{i=1}^n \lambda_i e_i^*$$ with $e^*_i \in \mathrm{Ext}(C)$ and the coefficients $\lambda_i \geq 0$ satisfying $\sum_{i=1}^n \lambda_i = 1$, such that$$\left\| y^* - \sum_{i=1}^n \lambda_i e_i^* \right\| < \frac{\varepsilon}{2}.$$
    Thus, we have $\left\|x^*-\sum_{i=1}^n\lambda_i e_i^*\right\|<{\varepsilon}$. We finish the proof since $\sum_{i=1}^n\lambda_i e_i^*$ belongs to $A$ by the same argument as above.
\end{proof}

Taking into account \cite[Theorem 1.1]{ruessstegal} which shows that 
$$\mathrm{Ext}\left(B_{\mathcal{I}(X,Y^*)}\right)= \mathrm{Ext}\left(B_{(X\widehat\otimes_\varepsilon Y)^*}\right)= \bigl\{x^*\otimes y^*\colon x^*\in\mathrm{Ext}(B_{X^*}),~y^*\in\mathrm{Ext}(B_{Y^*})\bigr\},$$
we get the following result.

\begin{theorem}\label{cor:wstar-density}
    Let $X$ and $Y$ be Banach spaces.
    \begin{itemize}
        \item[(1)] $\mathrm{NA}_\mathrm{nu}(X,Y^*)$ is $w^*$-dense in $\mathcal{I}(X,Y^*)$.
        \item[(2)] If $X\widehat\otimes_\varepsilon Y$ does not contain any isomorphic copy of $\ell_1$, then $\mathrm{NA}_\mathrm{nu}(X,Y^*)$ is norm-dense in $\mathcal{I}(X,Y^*)$.
    \end{itemize}
\end{theorem}

The next consequence of Lemma \ref{gen} will be performed in spaces of integral polynomials.
Taking into account that the inclusion
$$\mathrm{Ext}\left(B_{\mathcal{P}_{\mathcal{I}}\left(^NX\right)}\right)\subset\{\lambda\phi^N\colon \phi\in X^*,~\|\phi\|=1~\text{and}~\lambda\in\mathbb{K},\,|\lambda|=1\}$$
holds by \cite[Proposition 1]{BR}, we get the following theorem.

\begin{theorem}\label{cor:polywstar-density}
    Let $X$ be a Banach space and $N\in\mathbb{N}$. 
    \begin{itemize}
        \item[(1)] $\mathrm{NA}_\mathrm{nu}\left(^NX\right)$ is $w^*$-dense in $\mathcal{P}_{\mathcal{I}}\left(^NX\right)$.
        \item[(2)] If $\widehat\otimes_{\varepsilon,s,N}X$ does not contain any isomorphic copy of $\ell_1$, then $\mathrm{NA}_\mathrm{nu}\left(^NX\right)$ is norm-dense in $\mathcal{P}_{\mathcal{I}}\left(^NX\right)$.
    \end{itemize}
\end{theorem}

A couple of remarks are pertinent.

\begin{rem}~\label{remark:sec2}
\begin{itemize}
\item[(1)] {\slshape Theorem \ref{cor:wstar-density}(1) improves \cite[Theorem 4.7]{DGJR}, which previously required the additional assumption that either $X^*$ or $Y^*$ has the RNP.}
\item[(2)] {\slshape Theorem \ref{cor:polywstar-density}(1) refines \cite[Theorem 3.11]{DGJR}, where the same conclusion was obtained under the condition that the symmetric tensor product $\widehat\otimes_{\varepsilon,s,N}X$ contains no isomorphic copy of $\ell_1$.}
\item[(3)]  In \cite[Theorem 2]{BR}, it is shown that $\mathcal{P}_{\mathcal{I}}\left(^NX\right)$ and $\mathcal{P}_\mathrm{nu}\left(^NX\right)$ are isometrically isomorphic if the symmetric tensor product $\widehat\otimes_{\varepsilon,s,N}X$ does not contain any isomorphic copy of $\ell_1$. A similar argument allows us to prove that {\slshape $\mathcal{I}(X,Y^*)$ and $\mathcal{N}(X,Y^*)$ are isometrically isomorphic whenever $X\widehat\otimes_\varepsilon Y$ contains no isomorphic copy of $\ell_1$. }
\item [(4)] In \cite[Proposition 5]{Ros2}, it is proved that  every integral operator in $\mathcal{I}(X,Y^*)$ can be approximated in norm by finite rank operators whenever  $X\widehat\otimes_\varepsilon Y$ does not contain any isomorphic copy of $\ell_1$. {\slshape Theorem \ref{cor:wstar-density}(2) provides the stronger conclusion that every integral operator  can be approximated in norm by finite rank operators in $\mathrm{NA}_\mathrm{nu}(X,Y^*)$.}
\end{itemize}
\end{rem}

To obtain the result regarding projective tensor products, we recall two well-known results. First, if either $X^*$ or $Y^*$ has the AP, then $X^* \widehat{\otimes}_\pi Y^*$ is isometrically isomorphic to $\mathcal{N}(X, Y^*)$ \cite[Chapter 4.1]{Rya}. Second, if $X^*$ has the AP, then $\mathcal{P}_\mathrm{nu}\left(^NX\right)$ and $\widehat{\otimes}_{\pi,s,N} X^*$ are isometrically isomorphic \cite[p. 20]{BR}. Consequently, the following corollary follows immediately from Theorem \ref{cor:wstar-density}(2), Theorem \ref{cor:polywstar-density}(2), and Remark \ref{remark:sec2}(3).

\begin{cor}\label{Corollary:onedualAP-NAtensors-wstar-dense}
    Let $X$ and $Y$ be Banach spaces such that $X^*$ has the AP. 
    \begin{itemize}
        \item[(1)] If $X\widehat\otimes_\varepsilon Y$ does not contain any isomorphic copy of $\ell_1$, then the set of norm-attaining elements  is norm-dense in $X^*\widehat\otimes_\pi Y^*$.
        \item[(2)] If $\widehat\otimes_{\varepsilon,s,N}X$ does not contain any isomorphic copy of $\ell_1$, then the set of norm-attaining elements  is norm-dense in $\widehat\otimes_{\pi,s,N}X^*$.
    \end{itemize}
\end{cor}

Finally, we construct an element in $\ell_p \widehat\otimes_\pi \ell_q$ which does not attain its projective norm for $1 < p, q < \infty$ satisfying $\frac{1}{p} + \frac{1}{q} < 1$. To the best of the authors knowledge, this is the first such example within the projective tensor product of reflexive spaces (and such that the tensor product is reflexive too, see \cite[Corollary 4.24]{Rya} for instance).

\begin{example}\label{Example}
    For $1< p,q<\infty$ with $\frac{1}{p}+\frac{1}{q}<1$, let $(e_i)_i$ and $(f_i)_i$ be the canonical basis of $\ell_p$ and $\ell_q$, respectively. Then, for $r=\frac{pq}{p+q}>1$, 
    $$z\coloneqq\sum_{i=1}^\infty {2^{-\frac{i}{r}}} e_i\otimes f_i\in \ell_p\widehat\otimes_\pi \ell_q$$
    does not attain its projective norm. In particular, $\NA\left(\ell_p \pten \ell_q\right)\neq \ell_p \widehat\otimes_\pi \ell_q$.
\end{example}

\begin{proof} Let $r'$, $p'$, and $q'$ be the conjugate exponents of $r$, $p$ and $q$, respectively, and let  $(e_i^*)_i$ and $(f_i^*)_i$ be the canonical basis of $\ell_{p'}$ and $\ell_{q'}$, respectively. 

Consider the diagonal projection $P\colon \ell_p\widehat\otimes_\pi\ell_q\to\ell_p\widehat\otimes_\pi\ell_q$ defined by
    $$P( x\otimes y)=\sum_{i=1}^\infty e_i^*(x)f_i^*(y) e_i\otimes f_i,\quad \forall(x,y)\in\ell_{p}\times\ell_{q}.$$
    Note that $P$ is a projection with norm $1$, and that the diagonal $P\left(\ell_{p}\widehat\otimes_\pi \ell_{q}\right)$ is isometrically isomorphic to $\ell_{r}$ by the canonical identification \cite[Theorem~1.3]{AF}. Likewise, $P^*$ is also the diagonal projection and, by the paragraph after Definition~2.1 in \cite{CDSV}, we see that the diagonal $P^*\left(\ell_{p'}\widehat\otimes_\varepsilon\ell_{q'}\right)$ is isometrically isomorphic to $\ell_{r'}$ since it holds that $$p>q'\Longleftrightarrow q>p'\Longleftrightarrow\frac{1}{p}+\frac{1}{q}<1.$$
    We first see that $\|z\|_\pi=1$. Indeed, we have
    \begin{align*}
        \|z\|_\pi&=\sup_{\mu\in S}|\mu(z)|=\sup_{\mu\in S}|\mu(Pz)|=\sup_{\nu\in P^*\left(S\right)}|\nu(z)|=\sup_{c\in S_{\ell_{r'}}}\left|c\left(\left(2^{-\frac{i}{r}}\right)_i\right)\right|=\left\|\left(2^{-\frac{i}{r}}\right)_i\right\|_r=1
    \end{align*}
where $S=S_{\ell_{p'}\widehat\otimes_\varepsilon\ell_{q'}}$.

   Assume for the sake of contradiction that there exists an optimal representation$$z = \sum_{j=1}^\infty u_j \otimes v_j$$ such that $\sum_{j=1}^\infty \Vert{}u_j\Vert{}_p \Vert{}v_j\Vert{}_q = \Vert{}z\Vert{}_\pi = 1$ for some $u_j \in \ell_p$ and $v_j \in \ell_q$. For a set 
   $$\Lambda=\{j\in \mathbb{N}~:~\Vert{}u_j\Vert{}_p \Vert{}v_j\Vert{}_q\neq0\},$$
   we also have the representations $z = \sum_{j\in \Lambda} u_j \otimes v_j$ and $ \Vert{}z\Vert{}_\pi=\sum_{j \in \Lambda} \Vert{}u_j\Vert{}_p \Vert{}v_j\Vert{}_q$.
   
 Define an element $\zeta$ by
    $$\zeta=\sum_{i=1}^\infty2^{-\frac{i}{r'}}e^*_i\otimes f_i^*\in P^*\left(\ell_{p'}\widehat\otimes_\varepsilon\ell_{q'}\right).$$
    Since it holds that
    $$\|\zeta\|_\varepsilon=\left\|\left(2^{-\frac{i}{r'}}\right)_i\right\|_{r'}=1~\text{and}~\zeta(z)=\sum_{i=1}^\infty2^{-\frac{i}{r}-\frac{i}{r'}}=1,$$
    $\zeta$ attains its norm at $z$ as a functional on $\ell_p\widehat\otimes_\pi\ell_q$. Therefore, by \cite[Theorem 3.1]{DJRR} and the (generalized) H\"older inequality, we have 
    \begin{align}
        \|u_j\|_p\|v_j\|_q&=\zeta(u_j\otimes v_j) =\sum_{i=1}^\infty 2^{-\frac{i}{r'}} e_i^*(u_j)f_i^*(v_j)\leq\sum_{i=1}^\infty 2^{-\frac{i}{r'}}|e_i^*(u_j)||f_i^*(v_j)|\nonumber\\
        &\leq \left\|(2^{-\frac{i}{r'}})_i\right\|_{r'} \left\|\left(|e_i^*(u_j)||f_i^*(v_j)|\right)_i\right\|_r\leq \|u_j\|_p\|v_j\|_{q}.\label{eq:holdr}
    \end{align}
for every $j\in \Lambda$. Since all the inequalities in \eqref{eq:holdr} hold with equality, we have
\begin{equation*}
|e_i^*(u_j)| = 2^{-\frac{i}{p}} \|u_j\|_p \quad \text{and} \quad |f_i^*(v_j)| = 2^{-\frac{i}{q}} \|v_j\|_q
\end{equation*}
for every $i\in \mathbb{N}$ and $j\in \Lambda$. Indeed, together with the equality condition for the H\"older inequality \cite[p.~65]{Rud}, the first equality in \eqref{eq:holdr} implies
$$|e_i^*(u_j)||f_i^*(v_j)|=2^{-\frac{i}{r}}\|u_j\|_p\|v_j\|_q,$$
and the second one gives the existence of $\alpha>0$ such that
$$|e_i^*(u_j)|=\alpha |f_i^*(v_j)|^\frac{q}{p}$$
for every $i\in \mathbb{N}$ and $j\in \Lambda$. Combining these two, we get $$2^{-\frac{i}{r}}\|u_j\|_p\|v_j\|_q=\alpha |f_i^*(v_j)|^{\frac{q}{p}+1}=\alpha|f_i^*(v_j)|^{\frac{q}{r}}$$ 
which leads us to get
 $$|f_i^*(v_j)|=\left(\alpha^{-1}\|u_j\|_p\|v_j\|_q\right)^{\frac{r}{q}} 2^{-\frac{i}{q}}\quad\text{and}\quad|e_i^*(u_j)|=\left(\alpha^{-1}\|u_j\|_p\|v_j\|_q\right)^\frac{r}{p}2^{-\frac{i}{p}}.$$
Consequently, we have 
$$\|v_j\|_q^q=\sum_{i=1}^\infty |f_i^*(v_j)|^q=\left(\alpha^{-1}\|u_j\|_p\|v_j\|_q\right)^r\quad\text{and}\quad\|u_j\|_p^p=\sum_{i=1}^\infty |e_i^*(u_j)|^p=\left(\alpha^{-1}\|u_j\|_p\|v_j\|_q\right)^r$$
which shows
$$
|f_i^*(v_j)| = 2^{-\frac{i}{q}} \|v_j\|_q \quad \text{and} \quad |e_i^*(u_j)| = 2^{-\frac{i}{p}} \|u_j\|_p
$$
as desired.

Hence, the original representation $\sum_{i=1}^\infty {2^{-\frac{i}{r}}} e_i\otimes f_i$ of $z$ and the other one 
    \begin{align*}
        z=\sum_{j\in \Lambda} u_j\otimes v_j=\sum_{(k,l)}\left(\sum_{j\in \Lambda} e^*_k(u_j) f^*_l(v_j)\right) e_k\otimes f_l
    \end{align*}
 with respect to the canonical basis show that
\begin{align*}
&e^*_k\otimes f^*_l(z)=\sum_{i=1}^\infty {2^{-\frac{i}{r}}}e^*_k(e_i) f^*_l(f_i)={2^{-\frac{k}{r}}}\delta_{kl} \quad \text{and}\\
&e^*_k\otimes f^*_l(z)=\sum_{j\in \Lambda} e^*_k(u_j) f^*_l(v_j)=\sum_{j=1}^\infty 2^{-\frac{k}{p}-\frac{l}{q}}\|u_j\|_p\|v_j\|_q\varepsilon_{kj}\overline{\varepsilon_{lj}}
\end{align*}
for each $(k,l)\in \mathbb{N}^2$ where  $\delta_{kl}$ is the Kronecker delta and $\varepsilon_{ij}$ is defined by $$\varepsilon_{ij} \coloneqq \begin{cases} \mathrm{sign}(e_i^*(u_j)) & \text{if } (i,j) \in \mathbb{N} \times \Lambda, \\ 0 & \text{otherwise.} \end{cases}.$$
Note that $\varepsilon_{ij}=\overline{\mathrm{sign}(f_i^*(v_j))}$ whenever $j\in \Lambda$ since $e_i^*(u_j)f_i^*(v_j)=|e_i^*(u_j)||f_i^*(v_j)|$.

From $\sum_{j=1}^\infty \|u_j\|_p\|v_j\|_q=\sum_{j\in \Lambda} \|u_j\|_p\|v_j\|_q=1$, we see that $G_i\coloneqq(\sqrt{\|u_j\|_p\|v_j\|_q}\varepsilon_{ij})_{j}$ is an element in $\ell_2$ for all $i\in\mathbb{N}$. Moreover, it holds that
$$\left\langle G_k,G_l\right\rangle=\sum_{j=1}^\infty \|u_j\|_p\|v_j\|_q\varepsilon_{kj}\overline{\varepsilon_{lj}}=2^{\frac{k}{p}+\frac{l}{q}}e^*_k\otimes f^*_l(z)= {2^{-\frac{k}{r}}}2^{\frac{k}{p}+\frac{l}{q}}\delta_{kl}\quad \forall (k,l)\in \mathbb{N}^2$$
which implies$(G_i)_i$ is an orthonormal system in $\ell_2$. For the canonical basis $(g_{i})_i$ of $\ell_2$ and for any fixed $j_0\in \Lambda$, by the Bessel inequality \cite[Proposition~1.37]{FHHMPZ}, we have
\begin{equation}
    \sum_{k=1}^\infty|\langle g_{j_0},G_k\rangle|^2\leq \|g_{j_0}\|=1.\label{eq:contrad}
\end{equation}
On the other hand, for every $k\in \mathbb{N}$, we have$$\left\vert{} \langle g_{j_0}, G_k \rangle \right\vert{} = \left\vert{} \left\langle g_{j_0}, \left(\sqrt{\Vert{}u_j\Vert{}_p \Vert{}v_j\Vert{}_q} \varepsilon_{kj}\right)_{j \in \Lambda} \right\rangle \right\vert{} = \sqrt{\Vert{}u_{j_0}\Vert{}_p \Vert{}v_{j_0}\Vert{}_q}.$$
Thus, the left-hand side of \eqref{eq:contrad} diverges, which yields a contradiction.
\end{proof}

\begin{rem}\label{rem:mentionjung1}
The authors were recently aware of the recent preprint \cite{jung26}, where M. Jung proved that every infinite dimensional Banach space $X$ admits an equivalent renorming for which $X\pten X$ contains an element that does not attain its projective norm \cite[Corollary B]{jung26}.
\end{rem}

\section{II-polyhedral Banach spaces and nuclear norm-attainment}\label{Section 4}

Our main aim in this section is to show that $\mathrm{NA}\left(X^*\widehat\otimes_\pi Y^*\right) = X^*\widehat\otimes_\pi Y^*$ if $X$ is II-polyhedral and one of the spaces $X^*$ and $Y^*$ has the AP, and to provide consequences of this result. The methods are necessarily different from those in Section~\ref{Section 3}: while the density results there rely on convex combinations of extreme points, full norm-attainment requires a measure-theoretic analysis of integral representations.

\begin{definition}\label{defradonmeasureoperator}For $N\in \mathbb{N}$, let $\{X_i\}_{i=0}^N$ be a collection of Banach spaces, and let a subset $K \subset \prod_{i=0}^N B_{X_i^*}$ be compact in the product $w^*$-topology $\prod_{i=0}^N w^*$.
    A bounded $(N+1)$-multilinear form $T\in\mathcal{B}\left(\prod_{i=0}^N X_i\right)$ is said to be \emph{represented by a Radon measure} $\mu$ on $K$ if
    $$T\left((x_i)_{i=0}^N\right)=\int_{K} \prod_{i=0}^N x^*_i(x_i)\,d\mu\left((x_i^*)_{i=0}^N\right)\quad\forall (x_i)_{i=0}^N\in \prod_{i=0}^N B_{X_i^*}.$$
    Conversely, given a Radon measure $\mu$ on $K$, we denote by $T_\mu$ the $(N+1)$-multilinear form represented by $\mu$.
\end{definition}

For the measure $\mu$ in Definition \ref{defradonmeasureoperator}, we simply write $d\mu$ instead of $d\mu((x_i^* )_{i=0}^N)$ whenever there is no risk of confusion. We denote by $\tilde{\mu}$ the canonical extension of $\mu$ to the entire space $\prod_{i=0}^N B_{X_i^*}$, defined by$$\tilde{\mu}(E) = \mu\left(E \cap K\right)$$for any Borel set $E \subset \left(\prod_{i=0}^N B_{X_i^*}, \prod_{i=0}^N w^*\right)$. It is straightforward to see that $\mu$ and its canonical extension $\tilde{\mu}$ represent the same multilinear form. Thus, we will freely replace $\tilde{\mu}$ with $\mu$ whenever it is needed, without explicit mention. It also gives that $T_\mu \in \mathcal{I}\left(\prod_{i=0}^N{X_i}\right)$.

As a technical tool, we show that the support of a given Radon measure cannot be contained in a strictly smaller ball whenever its total variation coincides with the integral norm of the corresponding multilinear form.

\begin{lemma}\label{basic2}For $N\in \mathbb{N}$, let $\{X_{i}\}_{i=0}^N$ be a collection of Banach spaces, $r\in (0,1)$ be a real number, and $\mu$ be a Radon measure on $\left(\prod_{i=0}^N B_{X_i^*},\prod_{i=0}^N w^*\right)$ whose support is contained in $\left(rB_{X_{i_0}^*}\right)\times \prod_{i\neq i_0}B_{X_i^*}$, for some $0\leq i_0 \leq N$. If $T_\mu$ satisfies that $\|T_\mu\|_{\mathcal{I}}=|\mu|_\text{TV}$, then $\mu$ is the zero measure.
\end{lemma}

\begin{proof} Without loss of generality, we may assume that $i_0=0$. Let $H \colon \prod_{i=0}^N B_{X_i^*} \to rB_{X^*_0}\times\prod_{i=1}^N B_{X_i^*}$ be the map defined by$$H(x_0^*, x_1^*, \dots, x_N^*) = (rx_0^*, x_1^*, \dots, x_N^*)$$
for each $\left(x_0^*,x_1^*,\cdots,x_N^*\right)\in \prod_{i=0}^N B_{X_i^*}$.
We then define a measure $\nu$ on $\left(\prod_{i=0}^N B_{X_i^*},\prod_{i=0}^N w^*\right)$ by
$$\nu(E)=\mu(H(E))$$ for any Borel set $E \subset \left(\prod_{i=0}^N B_{X_i^*}, \prod_{i=0}^N w^*\right)$.

We first show 
$$\int_{\prod_{i=0}^N B_{X_i^*}} \chi_E\left(r^{-1}x_0^*,x_1^*,\cdots,x_N^*\right) \, d\mu=\int_{\prod_{i=0}^N B_{X_i^*}} \chi_E\left(x_0^*,\cdots,x_N^*\right) \, d\nu$$
for any Borel set $E \subset \left(\prod_{i=0}^N B_{X_i^*},\prod_{i=0}^N w^*\right)$ where $\chi_E$ is the characteristic function of $E$ defined on $r^{-1}B_{X^*_{0}}\times\prod_{i=1}^N B_{X_i^*}$.

 Indeed, it holds that 
$$\chi_{E}\left(r^{-1}x_0^*,x_1^*,\cdots,x_N^*\right)=\chi_{H(E)}\left(x_0^*,x_1^*,\cdots,x_N^*\right)$$
for any $\left(x_0^*,x_1^*,\cdots,x_N^*\right)\in \prod_{i=0}^N B_{X_i^*}$. Hence, we have that 

\begin{align*}
\int_{\prod_{i=0}^N B_{X_i^*}} \chi_E\left(r^{-1}x_0^*,x_1^*,\cdots,x_N^*\right) \, d\mu
&=\int_{\prod_{i=0}^N B_{X_i^*}} \chi_{H(E)}\left(x_0^*,x_1^*,\cdots,x_N^*\right) \, d\mu\\
&=\mu(H(E))=\nu(E)=\int_{\prod_{i=0}^N B_{X_i^*}} \chi_E\left(x_0^*,x_1^*,\cdots,x_N^*\right) \, d\nu.
\end{align*}
This leads us to get
$$ \int_{\prod_{i=0}^N B_{X_i^*}} r^{-1}\prod_{i=0}^Nx_i^*(x_i)\, d\mu=\int_{\prod_{i=0}^N B_{X_i^*}} \prod_{i=0}^Nx_i^*(x_i)\, d\nu$$
which is equivalent to 
$$r^{-1} T_\mu =T_\nu.$$
Therefore, we see that 
$$|\mu|_{TV}=\left\|T_\mu\right\|_{\mathcal{I}}=r\left\|T_\nu\right\|_{\mathcal{I}}\leq r|\nu|_{TV}.$$
Since $|\mu|_{TV}=|\nu|_{TV}$ by the construction of $\nu$, we get 
$\left\|T_\mu\right\|_{\mathcal{I}}=|\mu|_{TV}=0$.
\end{proof}

We are now ready to present the main result of this section in a general setting, from which the main theorem will be deduced.

\begin{theorem}\label{main}For $N\in\mathbb{N}$,  let $\{X_i\}_{i=0}^N$ be a collection of Banach spaces, and suppose $X_0$ is II-polyhedral.
If
$$\mathcal{N}\left(\prod_{i=1}^N {X_i}\right)\equiv\mathcal{I}\left(\prod_{i=1}^N{X_i}\right)~\text{isometrically and}~\mathcal{N}\left(\prod_{i=1}^N {X_i}\right)=\mathrm{NA}_\mathrm{nu}\left(\prod_{i=1}^N {X_i}\right),$$
then we have 
$$\mathcal{N}\left(\prod_{i=0}^N {X_i}\right)\equiv\mathcal{I}\left(\prod_{i=0}^N{X_i}\right)~\text{isometrically and}~\mathcal{N}\left(\prod_{i=0}^N {X_i}\right)=\mathrm{NA}_\mathrm{nu}\left(\prod_{i=0}^N {X_i}\right).$$
\end{theorem}

\begin{proof}
By the assumption, there exists a real number $r\in [0,1)$ such that 
$$\overline{\mathrm{Ext}(B_{X_0^*})}^{w^*} \subset \mathrm{Ext}(B_{X_0^*})\cup r B_{X_0^*}.$$
Define a $w^*$-compact set
$$K_{X^*_0}=\mathrm{Ext}(B_{X_0^*})\cup r B_{X_0^*}.$$
Since $\widehat{\otimes}_{\varepsilon,i=0}^N X$ is contained in $C\left(K_{X^*_0}\times \prod_{i=1}^N B_{X^*_i},\prod_{i=0}^N w^*\right)$ isometrically, for an integral multilinear form
$$T\in\mathcal{I}\left(\prod_{i=0}^N {X_i}\right)=\left(\widehat{\otimes}_{\varepsilon,i=0}^N X_i\right)^*,$$
there exists a Radon measure $\mu_{T}$ on $\left(K_{X^*_0} \times \prod_{i=1}^N B_{X^*_i},\prod_{i=0}^N w^*\right)$ such that 
$$T_{\mu_T}=T\quad\text{and}\quad\left|\mu_T\right|_\text{TV}=\|T\|_\mathcal{I}$$
 by the Hahn-Banach theorem.

~\\

\emph{Claim 1}. $\mu_T$ is supported on $\mathrm{Ext}(B_{X_0^*})\times\prod_{i=1}^N B_{X^*_i}$.\\

Since it is clear when $r=0$, we assume that $r>0$. 
Define measures $\mu_T^i$ ($i=1,2$) by restrictions
$$\mu_T^1= \mu_T\big|_{\mathrm{Ext}(B_{X_0^*})\times \prod_{i=1}^N B_{X^*_i}}\quad\text{and}\quad \mu_T^2=\mu_T \big|_{rB_{X_0^*}\times \prod_{i=1}^N B_{X^*_i}}.$$
Note that $\mathrm{Ext}(B_{X_0^*})\times \prod_{i=1}^N B_{X^*_i}$ is open (since every point $x^*\in \mathrm{Ext}(B_{X_0^*})$ is isolated in $K_{X^*_0}$) and $rB_{X_0^*}\times \prod_{i=1}^N B_{X^*_i}$ is closed.
Since the supports of $\mu_T^i$ are disjoint, $\mu_T$ and $\left|\mu_T\right|_\text{TV}$ can be expressed as
$$\mu_T=\mu_T^1+\mu_T^2~\text{and}~\left|\mu_T\right|_\text{TV}=\left|\mu_T^1\right|_\text{TV}+\left|\mu_T^2\right|_\text{TV}.$$

This gives that 
$$\|T\|_\mathcal{I}\leq \left\|T_{\mu_T^{1}}\right\|_\mathcal{I}+\left\|T_{\mu_T^{2}}\right\|_\mathcal{I}\leq \left|\mu_T^1\right|_\text{TV}+\left|\mu_T^2\right|_\text{TV}=\left|\mu_T\right|_\text{TV}=\|T\|_\mathcal{I}.$$
Hence, we have, for $i=1,2$,
$$\left\|T_{\mu_T^{i}}\right\|_\mathcal{I}=\left|\mu_T^i\right|_\text{TV}.$$
Thanks to Lemma \ref{basic2}, we see that $\mu_T^2=0$ which gives 
$$\mu_T=\mu_T^1.$$\\

\emph{Claim 2}. There exists a countable set $\Lambda\subset \mathrm{Ext}(B_{X_0^*})$ such that, for each $x^*\in \Lambda$, there exists a Radon measure $\mu_{x^*}$ on $\left(K_{X^*_0}\times \prod_{i=1}^N B_{X^*_i},\prod_{i=0}^N w^*\right)$ which is supported on $\{{x^*}\}\times \prod_{i=1}^N B_{X^*_i}$ satisfying
$$\mu_T=\sum_{{x^*}\in\Lambda} \mu_{x^*},\quad |\mu_T|_\text{TV}=\sum_{{x^*}\in\Lambda}\left|{\mu_{x^*}}\right|_\text{TV}, \quad\text{and}\quad\left\|T_{\mu_{x^*}}\right\|_{\mathcal{I}}=\left|\mu_{x^*}\right|_\text{TV}~\quad\forall{x^*}\in\Lambda.$$

Since every element of $\mathrm{Ext}(B_{X_0^*})$ is isolated in $K_{X^*_0}$, the set
$$\{x^*\}\times \prod_{i=1}^N B_{X^*_i}$$ 
is clopen in $K_{X^*_0}\times \prod_{i=1}^N B_{X^*_i}$ for every $x^*\in \mathrm{Ext}(B_{X_0^*})$, and all such sets are pairwise disjoint.

Since $\mu_T$ is a Radon measure, $\left|\mu_T\right|$ is also a Radon measure (see \cite[P222 Proposition 7.16]{Fol} for instance). 
Therefore, for each $n\in\mathbb{N}$, there exists a compact subset 
$$K_n\subset \mathrm{Ext}(B_{X_0^*})\times \prod_{i=1}^N B_{X^*_i}$$
such that 
$$\left|\mu_T\right|_\text{TV}-\frac{1}{n}\leq \left|\mu_T\right|(K_n).$$
Since the set $\left\{\{x^*\}\times \prod_{i=1}^N B_{X^*_i} \colon x^*\in \mathrm{Ext}(B_{X_0^*})\right\}$ is an open cover of $K_n$, there exists a finite set $F_n\subseteq \mathrm{Ext}(B_{X_0^*})$ such that
$$\left\{\{{x^*}\}\times \prod_{i=1}^N B_{X^*_i} \colon  {x^*}\in F_n\right\}$$
is a cover of $K_n$.
For a countable set $\Lambda:=\bigcup\limits_{n\in\mathbb N} F_n$, define a measure $\mu_{x^*}$ for every ${x^*}\in \Lambda$ by
\begin{equation*}
\mu_{x^*}(E):=\mu_T\left(E\cap\left({\{{x^*}\}\times \prod_{i=1}^N B_{X^*_i}}\right)\right),
\end{equation*}
for each Borel set $E\subset \left(K_{X^*_0}\times \prod_{i=1}^N B_{X^*_i},\prod_{i=0}^N w^*\right)$.
It is obvious that
$$\left|\mu_T\right|_\text{TV}=\sum_{{x^*}\in \Lambda}\left|\mu_T\right|\left(\{{x^*}\}\times \prod_{i=1}^N B_{X^*_i}\right)=\sum_{{x^*}\in\Lambda}|\mu_{x^*}|_\text{TV}.$$
Hence, we also deduce
$$\mu_T=\sum_{{x^*}\in\Lambda} \mu_{x^*}\quad\text{and}\quad \|T\|_\mathcal{I}\leq \sum_{{x^*}\in \Lambda} \left\|T_{\mu_{x^*}}\right\|_{\mathcal{I}}\leq \sum_{{x^*}\in \Lambda}\left|\mu_{x^*}\right|_\text{TV}=\left|\mu_T\right|_\text{TV}=\|T\|_\mathcal{I} $$
which proves Claim 2.\\

\emph{Claim 3}. For every ${x^*}\in \Lambda$, there exists a Radon measure $\nu_{x^*}$ on  $\left(\prod_{i=1}^N B_{X_i^*},\prod_{i=1}^N w^*\right)$ such that 
$$\mu_{x^*} = \delta_{x^*}\times \nu_{x^*}\quad \text{and}\quad\left|\nu_{x^*}\right|_\text{TV}=\left|\mu_{x^*}\right|_\text{TV}.$$ Here, $\delta_{x^*}$ is the Dirac measure at ${x^*}$ and the product measure $\delta_{x^*} \times \nu_{x^*}$ is defined by
$$\delta_{x^*} \times \nu_{x^*}(E)=\nu_{x^*}(E_{x^*})$$
for any Borel subset $E\subset \left(K_{X^*_0} \times \prod_{i=1}^N B_{X^*_i},\prod_{i=0}^N w^*\right)$ where 
$$E_{x^*}= \left\{ (y_i^*)_{i=1}^N \in \prod_{i=1}^N B_{X^*_i} \colon \left({x^*}, (y_i^*)_{i=1}^N\right) \in E\right\}.$$

Before we prove Claim 3, we first note that $\delta_{x^*} \times \nu_{x^*}$ is well defined in general. Indeed, for any Borel subset $E\subset \left(K_{X^*_0} \times \prod_{i=1}^N B_{X^*_i},\prod_{i=0}^N w^*\right)$, the intersection $E\cap \left(\{{x^*}\}\times \prod_{i=1}^N B_{X^*_i}\right)$ is again Borel. Since the map from $\prod_{i=1}^N B_{X^*_i}$ to the closed subspace $\{{x^*}\} \times \prod_{i=1}^N B_{X^*_i}$ given by 
$$(y_i^*)_{i=1}^N \mapsto \left({x^*}, (y_i^*)_{i=1}^N\right)$$
is a homeomorphism, it is clear that $E_{x^*}$ is Borel in $\left(\prod_{i=1}^N B_{X^*_i},\prod_{i=1}^N w^*\right)$.

Define the Radon measure $\nu_{x^*}$ on $\left(\prod_{i=1}^N B_{X^*_i},\prod_{i=1}^N w^*\right)$ by 
$$\nu_{x^*}(F)=\mu_{x^*}(\{{x^*}\}\times F).$$
for every Borel set $F\subset \left(\prod_{i=1}^N B_{X^*_i},\prod_{i=1}^N w^*\right)$. Then, we get that 
\begin{align*}
\delta_{x^*}\times \nu_{x^*}(E)
&=\nu_{x^*}(E_{x^*}) =\mu_{x^*}(\{{x^*}\}\times E_{x^*})\\
&=\mu_{x^*}\left(E\cap \left(\{{x^*}\}\times \prod_{i=1}^N B_{X^*_i}\right)\right)=\mu_{x^*}(E)
\end{align*}
 for any Borel subset $E\subset \left(K_{X^*_0} \times \prod_{i=1}^N B_{X^*_i},\prod_{i=0}^N w^*\right)$. This also clearly shows the equality 
$$\left|\nu_{x^*}\right|_\text{TV}=\left|\mu_{x^*}\right|_\text{TV}.$$\\

We now consider $Y:=\widehat\otimes_{\varepsilon,i=1}^N X_i$ as a subspace of $C\left(\prod_{i=1}^N B_{X^*_i},\prod_{i=1}^N w^*\right)$, and denote by $T_{x^*}$ the restriction $\nu_{x^*}\big|_Y$ of $\nu_{x^*}$ on $Y$. Note that $T_{x^*}\in {Y^*}\equiv\mathcal{I}\left(\prod_{i=1}^N {X_i}\right)\equiv\mathcal{N}\left(\prod_{i=1}^N {X_i}\right)$. By the assumption, for each ${x^*}\in\Lambda$, there exists $(\xi_{j,i,{x^*}}^*)_{(j,i)\in  \mathbb{N}\times\{1,\ldots, N\} }\subset X_i^*$ such that
$$T_{x^*}=\sum_{j}\otimes_{i=1}^N\xi_{j,i,{x^*}}^*\quad\text{and}\quad\|T_{x^*}\|_\mathcal{I}=\|T_{x^*}\|_\mathrm{nu}=\sum_{j}\prod_{i=1}^N\|\xi_{j,i,x^*}^*\|.$$\\

\emph{Claim 4}. For each ${x^*}\in \Lambda$, it holds that 
$$T_{\mu_{x^*}}=\sum_j{x^*}\otimes\left(\otimes_{i=1}^N \xi_{j,i,x^*}^*\right) \quad\text{and}\quad\left\|T_{\mu_{x^*}}\right\|_{\mathcal{I}}=\left\|T_{x^*}\right\|_\mathrm{nu}.$$\\

 For any finite sum $\sum_k \otimes_{i=0}^N x_{i,k}\in \widehat\otimes_{\varepsilon,i=0}^N X_i$, we have that 

\begin{align*}
T_{\mu_{x^*}}\left(\sum_k \otimes_{i=0}^N x_{i,k}\right)
&=\sum_kT_{\mu_{x^*}}\left(\otimes_{i=0}^N x_{i,k}\right)\\
&=\sum_k \int_{K_{X^*_0} \times \prod_{i=1}^N B_{X^*_i}}\prod_{i=0}^Nx_i^*(x_{i,k})\, d\mu_{x^*}\\
&=\sum_k\int_{K_{X^*_0} \times \prod_{i=1}^N B_{X^*_i}}\prod_{i=0}^N x_i^*(x_{i,k})\, d\delta_{x^*}\times\nu_{x^*}\\
&=\sum_k {x^*}(x_{0,k})\int_{\prod_{i=1}^N B_{X^*_i}}\prod_{i=1}^N x_i^*(x_{i,k})\, d\nu_{x^*}\\
&=\sum_k {x^*}(x_{0,k}) T_{x^*}\left(\otimes_{i=1}^N x_{i,k}\right)\\
&=\sum_k {x^*}(x_{0,k})\left(\sum_{j}\prod_{i=1}^N\xi_{j,i,x^*}^*(x_{i,k})\right) \\
&=\sum_j \left(\sum_{k}{x^*}(x_{0,k})\prod_{i=1}^N\xi_{j,i,x^*}^*(x_{i,k})\right)\\
&=\sum_j{x^*}\otimes\left(\otimes_{i=1}^N \xi_{j,i,x^*}^*\right)\left(\sum_k \otimes_{i=0}^N x_{i,k}\right).
\end{align*}
Hence, we deduce 
$$T_{\mu_{x^*}}=\sum_j{x^*}\otimes\left(\otimes_{i=1}^N \xi_{j,i,x^*}^*\right) \quad \text{and}\quad \left\|T_{\mu_{x^*}}\right\|_{\mathcal{I}}\leq \sum_{j}\prod_{i=1}^N\|\xi_{j,i,x^*}^*\|=\left\|T_{x^*}\right\|_\mathrm{nu}.$$
From Claim 2 and 3, we also get the reverse inequality
$$\left\|T_{\mu_{x^*}}\right\|_{\mathcal{I}}=\left|\mu_{x^*}\right|_\text{TV}=\left|\nu_{x^*}\right|_\text{TV}\geq\|T_{x^*}\|_\mathcal{I}=\left\|T_{x^*}\right\|_\mathrm{nu}$$
 which gives Claim 4.\\

Finally, we show $T\in \mathrm{NA}_\mathrm{nu}(\prod_{i=0}^N X_i)$. From the representation in Claim 2 and 4, we have that

$$T=\sum_{x^*\in\Lambda}\sum_j\left({x^*}\otimes\left(\otimes_{i=1}^N \xi_{j,i,x^*}^*\right)\right).$$

Moreover, we see that 

$$\|T\|_\mathrm{nu}\leq \sum_{x^*\in\Lambda}\sum_j\prod_{i=1}^N \|\xi_{j,i,x^*}^*\|=\sum_{x^*\in\Lambda}\left\|T_{x^*}\right\|_\mathrm{nu}=\sum_{x^*\in\Lambda}\left\|T_{\mu_{x^*}}\right\|_{\mathcal{I}}=\left\|T\right\|_{\mathcal{I}}\leq \|T\|_\mathrm{nu}$$ 
which finishes the proof.
\end{proof}

We are now ready to state and prove the main result of the section.

\begin{theorem}\label{theoremB}
Let $X$ and $Y$ be Banach spaces, and suppose $X$ is II-polyhedral. If one of the spaces $X^*$ or $Y^*$ has the AP, then $\mathrm{NA}\left(X^*\widehat\otimes_\pi Y^*\right)=X^*\widehat\otimes_\pi Y^*$.
\end{theorem}

\begin{proof}
This corresponds to the special case of Theorem \ref{main} for $N=1$.
\end{proof}

As an application to the classical theory of norm-attaining operators, we obtain Corollary \ref{cor:lineardense}  which is immediate using \cite[Corollary 3.11]{DJRR}. This provides new examples of pairs of Banach spaces satisfying the density of norm-attaining operators.

\begin{cor}\label{cor:lineardense}
Let $X$ and $Y$ be Banach spaces, and suppose $X$ is II-polyhedral. If one of the spaces $X^*$ or $Y^*$ has the AP,  then the sets of norm-attaining linear operators from $X^*$ to $Y^{**}$ and from $Y^*$ to $X^{**}$ are dense in the corresponding space of linear operators.
\end{cor}

\section{Proximinality and norm-attaining tensors}\label{Section 5}
The results of the previous sections establish conditions under which every nuclear operator attains its nuclear norm.
In this section, we explore a structural consequence of this phenomenon: the proximinality of $\mathrm{ker}(J)$ in the projective tensor product.
Recall that for a Banach space $X$ and a subspace $Y\subset X$, we say that $Y$ is \emph{proximinal} in $X$ if, given any $x\in X$, it follows that
$$\mathrm{dist}(x,Y)=\min\{\|x-y\|\colon y\in Y\}.$$
This is equivalent to the fact that the quotient map $\pi\colon X\to X/Y$ satisfies the following:
given $z+Y\in X/Y$ there exists $x\in X$ such that $\pi(x)=z+Y$ and $\|x\|=\|z+Y\|$.
Let us comment that, to the best of the authors knowledge, the question whether $\ker(J)$ is proximinal has not been previously considered in the literature.

The desired connection between the phenomenon of nuclear norm-attainment and the proximinality of $\mathrm{ker}(J)$ is described in the following proposition.

\begin{prop}\label{proximinal}
    Let $X$ and $Y$ be two Banach spaces.
    If $\mathrm{NA}_\mathrm{nu}(X,Y)=\mathcal{N}(X,Y)$, then $\mathrm{ker}(J)$ is a proximinal subspace of $X^*\widehat\otimes_\pi Y$.
\end{prop}

\begin{proof}
Select any
    $$z+\mathrm{ker}(J)\in\Big(\left(X^*\widehat\otimes_\pi Y\right)/\mathrm{ker}(J)\Big)\setminus \{0+\mathrm{ker}(J)\}.$$  
    Observe that there exist sequences $(x_i^*)_i$ in $X^*$ and $(y_i)_i$ in $Y$ such that
    $$z+\mathrm{ker}(J)=\sum_{i=1}^\infty (x_i^*\otimes y_i+\mathrm{ker}(J))\quad\text{and}\quad \|z+\mathrm{ker}(J)\|=\sum_{i=1}^\infty\|x_i^*\|\|y_i\|.$$
    
    It is immediate from the isometric isomorphism between $\mathcal{N}(X,Y)$ and $\left(X^*\widehat\otimes_\pi Y\right)/\mathrm{ker}(J)$ and the assumption $\mathrm{NA}_\mathrm{nu}(X,Y)=\mathcal{N}(X,Y)$.

    For the quotient map $\pi\colon X^*\widehat\otimes_\pi Y\to\left(X^*\widehat\otimes_\pi Y\right)/\mathrm{ker}(J)$, it is clear that
    $$\pi(w)=z+\mathrm{ker}(J)$$ for $w\coloneqq\sum_{i=1}^\infty x_i^*\otimes y_i\in X^*\widehat\otimes_\pi Y$.
    
    Hence, the proof will be finished when we have proved that $\|z+\mathrm{ker}(J)\|=\|w\|$.
    On the one hand, since $\pi$ is a norm-one operator, the inequality $\|z+\mathrm{ker}(J)\|\leq\|w\|$ is immediate.
    On the other hand, the triangle inequality implies
    $$\|w\|\leq\sum_{i=1}^\infty\|x_i^*\|\|y_i\|=\|z+\mathrm{ker}(J)\|,$$
    as desired.
\end{proof}

Let us show that, in order for $\mathrm{ker}(J)$ to be proximinal, the assumption that every nuclear operator attains its norm is not necessary.

\begin{example}\slshape 
    In \cite{DGJR,GGR,Rue}, there are examples of Banach spaces $X$ and $Y$ with the AP (and, consequently, $\mathrm{ker}(J)=\{0\}$, which is trivially proximinal) and such that $\mathrm{NA}\left(X^*\widehat\otimes_\pi Y\right)\neq X^*\widehat\otimes_\pi Y\equiv\mathcal{N}(X,Y)$ (for instance this happens for $X=\ell_2$ and $Y=c_0$, and we may also consider the new Example~\ref{Example}).
 \end{example}

However, under the hypothesis of full-attainment of projective tensor products, Proposition~\ref{proximinal} is an equivalence.

\begin{rem}
{\slshape     Whenever $\mathrm{NA}\left(X^*\widehat\otimes_\pi Y\right)=X^*\widehat\otimes_\pi Y$, $\ker(J)$ is proximinal if and only if $\mathrm{NA}_\mathrm{nu}(X,Y)=\mathcal{N}(X,Y).$}
    Indeed, if $\ker(J)$ is proximinal, for any $T\in\mathcal{N}(X,Y)$ and any tensor $\tilde{T}\in X^*\widehat\otimes_\pi Y$ with $J(\tilde{T})=T$, we can find a best approximation $v\in\ker(J)$ such that $\|\tilde{T}-v\|_\pi=\mathrm{dist}(\tilde T,\ker(J))$. Letting $u=\tilde{T}-v$, we have $J(u)=T$ and $\|u\|_\pi=\|T\|_\mathrm{nu}$. By the hypothesis, $u$ has an optimal representation $u=\sum_{i=1}^\infty x^*_i\otimes y_i$. Therefore, 
    $$\|T\|_\mathrm{nu}=\|J(u)\|_\mathrm{nu}\leq\sum_{i=1}^\infty\|x^*_i\|\|y_i\|=\|u\|_\pi=\|T\|_\mathrm{nu}.$$    
However, we do not know whether $\mathrm{NA}_\mathrm{nu}(X,Y)=\mathcal{N}(X,Y)$ whenever $\mathrm{NA}\left(X^*\widehat\otimes_\pi Y\right)=X^*\widehat\otimes_\pi Y$. 
\end{rem}  

Let us prove that $\ker(J)$ is proximinal in $X^*\pten Y^*$ when $X^*$ and $Y^*$ are separable.

\begin{theorem}\label{theo:mainproxi}
Let $X$ and $Y$ be Banach spaces. If $X^*$ and $Y^*$ are separable, then $\ker(J)$ is a proximinal subspace of $X^*\widehat\otimes_\pi  Y^*$.
\end{theorem}

We need some definitions and preliminary results. For Banach spaces $X$ and $Y$, from now on, we write $\Omega=(B_{X^*}\times B_{Y^*},w^*\times w^*)$ and consider the canonical map $\varphi\colon\Omega\to X^*\widehat{\otimes}_\pi Y^*$ defined by 
\begin{equation}\label{eq:varphi}\varphi(x^*,y^*)\coloneqq x^*\otimes y^*.\end{equation}

\begin{prop}\label{main-sec4}
Let $X$, $Y$ be Banach spaces and $\mu$ be a Radon measure which represents a nuclear operator $T\in\mathcal{N}(X,Y^*)$ (as an integral operator). If $\varphi$ is $\mu$-Bochner integrable, then the element $\displaystyle u=\int_{\Omega}\varphi\,d\mu\in X^*\widehat{\otimes}_\pi Y^*$ satisfies $Ju=T$.
    If, in addition, $|\mu|_\mathrm{TV}=\|T\|_\mathcal{I}$, then $\|u\|_\pi=|\mu|_\mathrm{TV}=\|T\|_\mathrm{nu}$.
\end{prop}

\begin{proof}
For $x\in X$ and $y\in Y$, let $e_{x,y}$ be the evaluation map on $X^*\widehat\otimes_\pi Y^*$ defined by
$$e_{x,y}(v)=\langle Jv(x),y\rangle$$
for all $v\in X^*\widehat\otimes_\pi Y^*$. Then, we have
    \begin{align*}
        \langle Ju(x),y\rangle=e_{x,y}(u)&=\int_\Omega e_{x,y}(\varphi(x^*,y^*))\,d\mu\\
        &=\int_\Omega e_{x,y}(x^*\otimes y^*)\,d\mu=\int_\Omega x^*(x)y^*(y)\,d\mu=\langle T(x),y\rangle.
    \end{align*}
    This shows that $Ju=T$.

    Moreover, thanks to the inequality $\|\varphi(x^*,y^*)\|_\pi\leq\|x^*\|\|y^*\|$, we have
    $$\|u\|_\pi=\left\|\int_\Omega \varphi\,d\mu\right\|_\pi\leq\int_\Omega\|x^*\otimes y^*\|_\pi\,d|\mu| \leq\int_\Omega1\,d|\mu|\leq|\mu|_\mathrm{TV}.$$
    Hence, under the additional assumption that $|\mu|_\mathrm{TV}=\|T\|_\mathcal{I}$, the chain of inequalities 
    $$\|T\|_\mathrm{nu}\leq\|u\|_\pi\leq|\mu|_\mathrm{TV}=\|T\|_\mathcal{I}\leq\|T\|_\mathrm{nu}$$
    produces the equality $\|u\|_\pi=|\mu|_\mathrm{TV}=\|T\|_\mathrm{nu}$.
\end{proof}

The next lemma guarantees that $\varphi$ is always $\mu$-Bochner integrable in our context of working with separable dual spaces.

\begin{lemma}\label{lemma-sec4-integrable}
    Let $X$ and $Y$ be Banach spaces. If $X^*$ and $Y^*$ are separable, then the map $\varphi$ is $\mu$-Bochner integrable for any finite Radon measure $\mu$ on $\Omega$.
\end{lemma}

\begin{proof}
By \cite[Proposition 6.5]{BCJ}, the Borel $\sigma$-algebra on $(B_{X^*},\|\cdot\|)$ and the one on $(B_{X^*},w^*)$ coincide (and, thus, we have the same for $B_{Y^*}$). Hence, by \cite[Proposition 1.2]{Fol}, the Borel $\sigma$-algebra on $\Omega$ is actually the same as the one on $(B_{X^*}\times B_{Y^*},\|\cdot\|_{X^*}\times\|\cdot\|_{Y^*})$.
Since $\varphi$ is continuous with respect to the norm-topology and $X^*\widehat{\otimes}_\pi Y^*$ is separable, it is $\mu$-measurable and, consequently, it is $\mu$-Bochner integrable for any finite Radon measure $\mu$.
\end{proof}

Now, we are ready to provide the pending proof.

\begin{proof}[Proof of Theorem~\ref{theo:mainproxi}]
Let $\Phi\colon X\widehat\otimes_\varepsilon Y\hookrightarrow C(\Omega)$ be the natural isometric inclusion map defined in the introduction. Note that $\Phi^*\colon \mathcal{M}(\Omega)\to\mathcal{I}(X,Y^*)$ is a quotient map and $\mathrm{ker}(\Phi^*)$ is proximinal.

For a nuclear operator $T\in \mathcal{N}(X,Y^*)\equiv\mathcal{I}(X,Y^*)$ (the later isometric equivalence holds because $Y^*$ has the RNP \cite[p.~524, Corollary 1]{DeFlo}), there is a Radon measure $\mu$ representing $T$ such that $|\mu|_\mathrm{TV}=\|T\|_\mathcal{I}$ by the Hahn-Banach theorem. By  Proposition \ref{main-sec4} and Lemma \ref{lemma-sec4-integrable}, an element $u\in X^*\widehat{\otimes}_\pi Y^*$ defined by
$$u=\int_\Omega\varphi\,d\mu,$$
satisfies
\[Ju=T \quad \text{and}\quad \|u\|_\pi=|\mu|_\mathrm{TV}=\|T\|_\mathrm{nu}.\qedhere\]
\end{proof}

According to \cite[Definition 2.1]{ADGJR}, given two Banach spaces $X$ and $Y$, it is said that an element $u\in X\widehat\otimes_\pi Y$ is an \emph{integral projective norm-attaining tensor} if there exists a finite positive Borel measure $\mu$ on $B_X\times B_Y$ with $|\mu|_\mathrm{TV}=\|u\|_\pi$ such that the map $\psi\colon B_X\times B_Y \to X\widehat{\otimes}_\pi Y$ defined by 
$$\psi(x,y)\coloneqq x\otimes y$$
is $\mu$-Bochner integrable and $u$ can be represented by the Bochner integral
$$u=\int_{B_X\times B_Y} \psi\,d\mu.$$ 

In the case that $X$ and $Y$ are dual spaces, observe that $\psi$ is exactly the same as the map $\varphi$ given in \eqref{eq:varphi} for dual spaces.
We write $\mathrm{INA}\left(X\widehat\otimes_\pi Y\right)$ to denote the subset of $X\widehat\otimes_\pi Y$ of all integral projective norm-attaining tensors.
It is known (and clear) that the inclusion $\mathrm{NA}\left(X\widehat\otimes_\pi Y\right)\subseteq \mathrm{INA}\left(X\widehat\otimes_\pi Y\right)$ always holds, and whether it is an equality was posed as an open problem in \cite[Question~2.2]{ADGJR}. (In Corollary~\ref{corollary:NA-neq-INA}, we will provide a negative answer to this question.)

More in general, the authors considered in \cite[Section 3]{ADGJR} the following notion: given a topology $\tau$ on $B_X\times B_Y$, it is said that $u$ is \emph{$\tau$ integral projective norm-attaining} if the measure $\mu$ above is $\tau$-Borel, and denoted by  $\mathrm{INA}_\tau\left(X\widehat\otimes_\pi Y\right)$ the set of all $\tau$ integral projective norm-attaining tensors.

From the proofs of Theorem~\ref{theo:mainproxi} and Lemma~\ref{lemma-sec4-integrable} the following result is immediate. 

\begin{cor}\label{cor:INA}
Let $X$ and $Y$ be Banach space. If $X^*$ and $Y^*$ are separable and one of the spaces $X^*$ or $Y^*$ has the AP, then $X^*\widehat\otimes_\pi Y^*=\mathrm{INA}\left(X^*\widehat{\otimes}_\pi Y^*\right)$.
\end{cor}

\begin{rem}~
\begin{itemize}
\item[(1)]
{\slshape The condition in Corollary~\ref{cor:INA} that both dual spaces $X^*$ and $Y^*$ are separable cannot be weakened to the separability of just one of the duals.}\ Indeed, \cite[Example~4.1]{ADGJR} showed that $\mathrm{INA}\left(C(\mathbb{T})^*\widehat\otimes_\pi\ell_2^2\right)\neq C(\mathbb{T})^*\widehat\otimes_\pi\ell_2^2$.  
\item[(2)] Analogously, {\slshape the condition in Corollary~\ref{cor:INA} that both spaces are dual spaces cannot be weakened to the fact that just one of them is a dual space.}\ Indeed, \cite[Example~4.1]{ADGJR} showed that $\mathrm{INA}\left(L_1(\mathbb{T}\right) \widehat\otimes_\pi\ell_2^2)\neq L_1(\mathbb{T}) \widehat\otimes_\pi\ell_2^2$.  
\end{itemize}
\end{rem}

Corollary~\ref{cor:INA} improves \cite[Proposition~3.2]{ADGJR} in which the reflexivity of one of the spaces $X$ and $Y$ is required and that the thesis of such proposition is that $X^*\pten Y^*=\mathrm{INA}_{w^*}\left(X^*\widehat{\otimes}_\pi Y^*\right)$.

To finish the paper, we show that Theorem~\ref{theo:mainproxi} and Example~\ref{Example} give a negative solution to \cite[Question~2.2]{ADGJR} of whether the sets $\NA\left(X\widehat\otimes_\pi Y\right)$ and $\mathrm{INA}\left(X\widehat\otimes_\pi Y\right)$ always coincide.

\begin{cor}\label{corollary:NA-neq-INA}
For $1< p,q<\infty$ with $\frac{1}{p}+\frac{1}{q}<1$, it holds that
$$
\NA\left(\ell_p \pten \ell_q\right)\neq \ell_p \widehat\otimes_\pi \ell_q \qquad \text{while} \qquad \mathrm{INA}\left(\ell_p\widehat\otimes_\pi\ell_q\right)= \ell_p\widehat\otimes_\pi\ell_q.
$$
\end{cor}

\begin{rem}\label{rem:mentionjung2}
The authors were recently aware of the recent preprint \cite{jung26}, where M. Jung solved independently the above mentioned \cite[Question~2.2]{ADGJR}. Indeed, he proved that every infinite dimensional Banach space $X$ admits an equivalent renorming for which $\mathrm{INA}(X\pten Y)\setminus \NA(X\pten X)$ is non-empty.
\end{rem}

\subsection*{Acknowledgements}
The authors are deeply grateful to Sheldon Dantas for pointing out a mistake in a previous version of the paper. They also thank him for fruitful comments.
Manwook Han thanks the University of Granada for its hospitality during his visit in November-December 2025, and the ``Maria de Maeztu'' Excellence Unit IMAG for its partial support, which covered his local expenses during this stay.

\subsection*{Funding}
Han and Kim were supported by the National Research Foundation of Korea(NRF) grant funded by the Korea government(MSIT) [NRF-2020R1C1C1A01012267]. The research of Martín and Rueda-Zoca has been supported by MICIU/AEI/10.13039/501100011033 and ERDF/EU through the grants PID2021-122126NB-C31, by ``Maria de Maeztu'' Excellence Unit IMAG, funded by MICIU/AEI/10.13039/501100011033 with reference CEX2020-001105-M, and Junta de Andaluc\'ia, grant FQM-0185. Part of this work was done while Manwook Han visited the University of Granada in November-December 2025, partially supported by the ``Maria de Maeztu'' Excellence Unit IMAG. 

\subsection*{Competing Interests} The authors have no relevant financial or non-financial interests to disclose.

\subsection*{Author Contributions} All authors contributed equally to the whole part of works together such as the study conception, design, and writing. All authors read and approved the final manuscript.


\end{document}